\documentclass[a4paper, 11pt,reqno]{amsart}

\usepackage{amsmath}
\usepackage{amssymb}
\usepackage{hyperref} 
\usepackage{amsfonts}
\usepackage{mathtools}
\usepackage{theoremref}
\usepackage{amsthm}
\usepackage[dvipsnames]{xcolor}
\usepackage[utf8]{inputenc}
\usepackage{array}
\usepackage{enumitem}
\usepackage[english]{babel}
\usepackage{comment}
\usepackage{multicol}

\calclayout

\setlist{topsep=0mm,partopsep=0mm,itemsep=1mm}

\theoremstyle{plain}
\newtheorem{lemma}{Lemma}[section]
\newtheorem{thm}[lemma]{Theorem}
\newtheorem{cor}[lemma]{Corollary}
\newtheorem{prop}[lemma]{Proposition}
\newtheorem*{mainthm}{Main Theorem}

\theoremstyle{definition}

\theoremstyle{remark}
\newtheorem{cclaim}{Claim}[lemma]

\newenvironment{claim}{\begin{cclaim}\it}{\end{cclaim}}

\newcommand{\N}{\mathbb{N}}
\newcommand{\Z}{\mathbb{Z}}

\newcommand{\D}{\mathcal{D}}

\newcommand{\J}{\mathcal{J}}

\newcommand{\subseteqcf}{\subseteq_{\textup{cf}}}

\newenvironment{thmenumerate}{\begin{enumerate}[label=\textup{(\roman*)},leftmargin=10mm]}{\end{enumerate}}
\newenvironment{nitemize}{\begin{itemize}[label=\textbullet, leftmargin=5mm]}{\end{itemize}}

\makeatletter
\@namedef{subjclassname@2020}{\textup{2020} Mathematics Subject Classification}
\makeatother

\begin{document}

\title[Semigroup subdirect products with $\Z$]{A characterisation of semigroups with only countably many subdirect products with \boldmath{$\Z$}}
\author[A.\ Clayton]{Ashley Clayton}
\address{School of Mathematics and Statistics, University of St Andrews, St Andrews, Scotland, UK}
\email{ac323@st-andrews.ac.uk}

\author[C.\ Reilly]{Catherine Reilly}
\address{School of Mathematics, University of East Anglia, Norwich NR4 7TJ, England, UK}
\email{C.Reilly@uea.ac.uk}

\author[N.\ Ru\v{s}kuc]{Nik Ru\v{s}kuc}
\address{School of Mathematics and Statistics, University of St Andrews, St Andrews, Scotland, UK}
\email{nr1@st-andrews.ac.uk}


\begin{abstract}
Let $\Z$ be the additive (semi)group of integers. We prove that for a finite semigroup $S$ the direct product $\Z\times S$ contains only countably many subdirect products (up to isomorphism) if and only if $S$ is regular.
As a corollary we show that $\Z\times S$ has only countably many subsemigroups (up to isomorphism) if and only if $S$ is completely regular.
\end{abstract}

\thanks{}
\maketitle

\section{Introduction}

For two semigroups $S$ and $T$, a subsemigroup $U\leq S\times T$ of their direct product is called a \emph{subdirect product} if it projects onto both $S$ and $T$, i.e. if $\{s\colon (s,t)\in U \text{ for some } t\}=S$ and $\{ t\colon (s,t)\in U \text{ for some } s\in S\}=T$.
Subdirect products are one of the fundamental concepts in general algebra (e.g.\ see \cite[Section II.8 ff.]{BS:CUA}), and have been extensively used in combinatorial group theory 
(for examples see \cite{baumslaugandroseblade,bridson13,bridsonmiller,Grunewald78,Mihailova66}), with more recent work in combinatorial semigroup theory \cite{Cl:FPFP20,CRR:paper1,acnr1} and elsewhere \cite{DMR:BHF20,pmnr}.

Let $\N$ denote the additive semigroup of positive integers.
In \cite[Theorem E]{acnr1} an intriguing link is established between the number of subdirect products inside $\N\times S$, where $S$ is finite, and algebraic properties of $S$. Specifically, it is shown that the following are equivalent: (i) $\N\times S$ contains only countably many subdirect products; (ii) $\N\times S$ contains only countably many non-isomorphic subdirect products; and (iii) for every $s\in S$ there exists $t\in S$ such that at least one of $st=s$ or $ts=s$ holds.

The main result in this paper concerns the same situation, but with the (semi)group $\Z$ of additive integers replacing $\N$, and featuring one of the fundamental semigroup-theoretic properties.

\begin{mainthm}
\label{thm:main}
The following are equivalent for a finite semigroup $S$:
\begin{thmenumerate}
\item \label{it:main1}
$S$ is regular;
\item\label{it:main2}
$\Z\times S$ contains only countably many subdirect products;
\item\label{it:main3}
$\Z\times S$ contains only countably many subdirect products up to isomorphism.
\end{thmenumerate}
\end{mainthm}

Recall that a semigroup $S$ is said to be \emph{regular} if each $x\in S$ is \emph{regular}, which means that  $xyx=x$ for some $y\in S$.
This $y$ can in fact be chosen to also satisfy $yxy=y$, in which case it is called a \emph{generalised inverse} of $x$.
For background on regularity see \cite[Section 3.4]{Ho:FST} or any other standard monograph on semigroup theory.

It is perhaps curious to contrast the above results with the situation in groups. 
Due to the equivalence between subdirect products in $G\times H$ and fiber products (Goursat's Lemma, \cite[Theorem 5.5.1]{HallBook76}), if $G$ and $H$ are finitely generated groups with only countably many subgroups, then $G\times H$ also has only countably many subdirect products (and even subgroups). Fleisher's Lemma \cite[Lemma IV.10.1]{BS:CUA} vastly extends the scope of Goursat's Lemma to arbitrary congruence permutable varieties, and so the previous observation holds for many algebraic structures beyond groups, e.g.\ rings, associative and Lie algebras, loops, etc., but not for semigroups.

The brunt of the paper is devoted to the proof of the Main Theorem. Specifically:
\begin{nitemize}
\item
\ref{it:main2}$\Rightarrow$\ref{it:main3} is obvious.
\item 
\ref{it:main1}$\Rightarrow$\ref{it:main2} is proved in Section \ref{sec:Sreg}: it is an immediate consequence of Theorem \ref{th:fg} which asserts that, when $S$ is finite and regular, every subdirect product in $\Z\times S$ is finitely generated.
\item 
\ref{it:main3}$\Rightarrow$\ref{it:main1} is proved in Section \ref{sec:notreg}, by considering an arbitrary finite non-regular $S$
and  constructing an uncountable family of pairwise non-isomorphic subdirect products of $\Z \times S$. 
\end{nitemize}

As a corollary of the Main Theorem we show that $\Z\times S$ has only countably many subsemigroups (up to isomorphism) if and only if $S$ is completely regular, i.e. a union of groups (Corollary \ref{co:subs}).

\section{Preliminaries}
\label{sec:prelim}

The paper does not require much background in semigroup theory. However, we will make extensive use of Green's $\J$-relation, which in a natural way reflects the ideal structure of a semigroup.
We review the basic definitions and properties that we require, and for a more systematic account refer the reader to any standard textbook on semigroup theory, such as \cite{Ho:FST}.

Let $S$ be a semigroup, and denote by $S^1$ the semigroup $S$ with an identity element adjoined to it if $S$ does not already have one.
For elements $x,y\in S$ we say that $x\leq_\J y$ if the ideal generated by $x$ is contained within that generated by $y$.
This is equivalent to $u_1yu_2=x$ for some $u_1,u_2\in S^1$.
The relation $\leq_\J$ is reflexive and transitive, but not necessarily symmetric, i.e. it is a pre-order. 
Associated to the pre-order $\leq_\J$ is  the equivalence $\J$ defined by $x\J y$ if and only if $x\leq_\J y$ and $y\leq_\J x$. The equivalence class of an element $x\in S$ is called the $\J$-\emph{class} of $x$ and is denoted by $J_x$.
The pre-order~$\leq_\J$ also induces a partial order on the set $S/\J$ of $\J$-classes via
$J_x\leq J_y$ if and only if~$x\leq_\J y$.

We will not require other Green's equivalences, but we will use some facts about $\J$-classes on a finite semigroup $S$, which follow because it is equal to Green's equivalence~$\D$ in this case \cite[Proposition 2.1.4]{Ho:FST}.
Specifically, assuming $S$ is finite, we have:

\begin{enumerate}[label=\textsf{(J\arabic*)},leftmargin=10mm]
\item\label{it:J1}
Any $\J$-class of $S$ either consists entirely of regular elements, or else entirely of non-regular elements
\cite[Proposition 2.3.1]{Ho:FST}.
(Thus we will talk of \emph{regular} and \emph{non-regular $\J$-classes}.)
\item\label{it:J2}
A regular $\J$-class contains an idempotent \cite[Proposition 2.3.2]{Ho:FST}.
\item\label{it:J3}
For a regular $\J$-class $J$ and any $x,y\in J$ 
there exist $u_1,u_2,v_1,v_2\in J$ such that $u_1xu_2=y$ and $v_1yv_2=x$
\cite[Propositions 2.3.2, 2.3.3]{Ho:FST}.
\item\label{it:J5}
If $J$ is a non-regular $\J$-class and if $x,y\in J$ then $J_{xy}< J$
\cite[Theorem 3.1.6]{Ho:FST}.
\item\label{it:J4}
$S$ has a unique minimal ideal, which is a regular $\J$-class
\cite[Proposition 3.1.4, Theorem 3.3.2]{Ho:FST}.
\end{enumerate}

\section{Regular $S$}
\label{sec:Sreg}

As explained in the Introduction, the pathway to establishing the implication \ref{it:main1}$\Rightarrow$\ref{it:main2} in the Main Theorem, is via the following:

\begin{thm}
\label{th:fg}
If $S$ is a finite regular semigroup then every subdirect product in $\Z\times S$ is finitely generated.
\end{thm}

The proof of Theorem \ref{th:fg} is the content of the remainder of this section.
We proceed via a series of lemmas. Throughout, $S$ is assumed to be finite and regular, and $T$ is a subdirect product in $\Z\times S$.
For $x\in S$ we let
\[
T_x:=\{ n\in\Z\::\: (n,x)\in T\}.
\]
Notice that $T_x\neq\emptyset$ because $T$ is subdirect, and that
$T=\bigcup_{x\in S} T_x$.
For sets $A,B$ we write $A\subseteqcf B$ to mean $A\subseteq B$ and $|B\setminus A|<\infty$.
We use $\N_0$ to denote the set of non-negative integers.

\begin{lemma}
\label{la:Te}
If $e\in S$ is an idempotent then $T_e$ is a subsemigroup of $\Z$, and hence precisely one of the following holds:
\vspace{-2mm}
\begin{multicols}{2}
\begin{thmenumerate}
\item \label{it:Te1}
$T_e=\{0\}$;
\item \label{it:Te2}
$T_e=d\Z$ for some $d>0$;
\item \label{it:Te3}
$T_e\subseteqcf d\N_0$ for some $d>0$;
\item \label{it:Te4}
$T_e\subseteqcf -d\N_0$ for some $d>0$.
\end{thmenumerate}
\end{multicols}
\vspace{-3mm}
In particular, $T_e$ is finitely generated.
\end{lemma}

\begin{proof}
For $m,n\in T_e$ we have $(m,e),(n,e)\in T$, hence
\[
T\ni (m,e)(n,e)=(m+n,e^2)=(m+n,e),
\]
which implies $m+n\in T_e$, and therefore $T_e$ is indeed a subsemigroup of $\Z$.
That one of \ref{it:Te1}--\ref{it:Te4} holds now follows from well known facts about subsemigroups of $\Z$.
Indeed, if $T_e$ contains both positive and negative numbers then $T_e$ is in fact a \emph{subgroup} of $\Z$, and so $T_e=d\Z$ for 
$d:=\gcd(T_e)$.
If $T_e$ contains no negative numbers, but does contain some positive numbers, then $T_e$ is in fact a non-trivial subsemigroup of $\N_0$, and it is well known that $T_e\subseteqcf d\N_0$ with $d:=\gcd(T_e)$; see 
\cite[Proposition 2.2]{RGS:numerical}.
The case where $T_e$ contains some negative numbers but no positive numbers is dual, and we get $T_e\subseteqcf -d\N_0$.
Finally, when $T_e$ contains neither positive nor negative numbers then $T_e=\{0\}$.
That $T_e$ is finitely generated in each of the four alternatives is straightforward; see \cite{SS:subsemigroups} or \cite[Theorem 2.7]{RGS:numerical} for the cases \ref{it:Te3}, \ref{it:Te4}.
\end{proof}

\begin{lemma}
\label{la:Tx}
Let $x\in S$ be arbitrary, let $y\in S$ be any generalised inverse of $x$, and let $e:=xy$.
Then there exists $r\in T_x$ such that one of the following holds:
\begin{thmenumerate}
\item \label{it:Tx1}
$T_e=\{0\}$ and $T_x=\{r\}$;
\item \label{it:Tx2}
$T_e=d\Z$ ($d>0$)  and $T_x=r+d\Z$;
\item \label{it:Tx3}
$T_e\subseteqcf d\N_0$ ($d>0$) and $T_x\subseteqcf r+d\N_0$;
\item \label{it:Tx4}
$T_e\subseteqcf -d\N_0$ ($d>0$) and $T_x\subseteqcf r-d\N_0$.
\end{thmenumerate}
\end{lemma}

\begin{proof}
As we already observed, $T_x\neq\emptyset$, and hence there exists some $r\in T_x$. 
From $ex=xyx=x$ we have $T_e+T_x\subseteq T_x$.
In particular,
\begin{equation}
\label{eq:rTe}
r+T_e=T_e+r\subseteq T_x.
\end{equation}
From $xy=e$ it follows that $T_x+T_y\subseteq T_e$.
In particular, fixing any $s\in T_y$, we have
\begin{equation}
\label{eq:sTx}
s+T_x=T_x+s\subseteq T_e.
\end{equation}

Now we examine in turn each of the cases \ref{it:Te1}--\ref{it:Te4} arising from Lemma \ref{la:Te}.
\vspace{1mm}

\ref{it:Tx1}
If $T_e=\{0\}$, then $|T_x|=1$ by \eqref{eq:sTx}, i.e. $T_x=\{r\}$.
\vspace{1mm}

\ref{it:Tx2}
Suppose $T_e=d\Z$.
Then \eqref{eq:rTe} implies $r+d\Z\subseteq T_x$.
For the reverse inclusion, let $r_1\in T_x$.
By \eqref{eq:sTx} we have $s+r,s+r_1\in T_e=d\Z$, and hence
$s+r\equiv 0\equiv s+r_1\pmod{d}$. Therefore $r\equiv r_1\pmod{d}$, and so $r_1\in r+d\Z$, as required.
\vspace{1mm}

\ref{it:Tx3}
Suppose $T_e\subseteqcf d\N_0$.
Note that \eqref{eq:sTx} implies that $T_x$ is bounded below. So, in this case we will take $r:=\min(T_x)$.
Now, $T_x\subseteq r+d\N_0$ is proved in exactly the same way as in \ref{it:Tx2}.
From $T_e\subseteqcf d\N_0$ it follows that $r+T_e\subseteqcf r+d\N_0$, which, combined with \eqref{eq:rTe}, implies
$T_x\subseteqcf r+d\N_0$, as required.

\ref{it:Tx4} This is dual to \ref{it:Tx3}.
\end{proof}

\begin{lemma}
\label{la:finF}
With $x$ and $e$ as in Lemma \ref{la:Tx},
the set $T_x\setminus (r+T_e)$ is finite.
\end{lemma}

\begin{proof}
We examine each of the cases \ref{it:Te1}--\ref{it:Te4} from Lemma \ref{la:Te}, together with the matching case from Lemma \ref{la:Tx}.
For \ref{it:Te1} and \ref{it:Te2} we have $T_x\setminus (r+T_e)=\emptyset$, for
\ref{it:Te3} 
\[
T_x\setminus (r+T_e)\subseteq (r+d\N_0)\setminus (r+T_e)\subseteq r+(d\N_0\setminus T_e),
\]
which is finite because $T_e\subseteqcf d\N_0$, and 
\ref{it:Te4} is dual.
\end{proof}

\begin{lemma}
\label{la:finA}
For every $x\in S$ there exists a finite set $A\subseteq T$ such that $T_x\times\{x\}\subseteq \langle A\rangle$.
\end{lemma}

\begin{proof}
By Lemma \ref{la:Te} there exists a finite set $B\subseteq \Z$ such that $T_e=\langle B\rangle$.
As $T_e\times\{e\}\cong T_e$, we have $T_e \times \{e\}=\langle B\times\{e\}\rangle$.
Let $F:=T_x\setminus (r+T_e)$, which is a finite set by Lemma~\ref{la:finF}.
Then
$T_x=F\cup (r+T_e)$, and hence
\begin{align*}
T_x\times\{x\} &= F\times\{x\} \cup (T_e\times\{e\})\cdot (r,x)
\subseteq \bigl\langle \{ (r,x)\} \cup (F\times \{x\})\cup (T_e\times\{e\})\bigr\rangle\\
&=\bigl\langle \{(r,x)\}\cup (F\times\{x\})\cup (B\times \{e\})\bigr\rangle.
\end{align*}
Since the set $\{(r,x)\}\cup (F\times\{x\})\cup (B\times \{e\})$ is finite, the lemma is proved.
\end{proof}

\begin{proof}[Proof of Theorem \ref{th:fg}]
The theorem follows immediately from $T=\bigcup_{x\in S} T_x$, finiteness of $S$, and the fact that each $T_x$ is contained in a finitely generated subsemigroup by Lemma \ref{la:finA}.
\end{proof}

\begin{cor}
\label{co:1imp2}
If $S$ is a finite regular semigroup, then $\Z\times S$ contains only countably many subdirect products.
\end{cor}

\begin{proof}
The semigroup $\Z\times S$ is countable, and each subdirect product contained in it is generated by a finite set by Theorem \ref{th:fg}.
\end{proof}

\section{Non-regular $S$}
\label{sec:notreg}

This section is entirely devoted to proving the following:

\begin{prop}\label{pro:nonreg}
If $S$ is a finite non-regular semigroup then $\Z\times S$ contains uncountably many pairwise non-isomorphic subdirect products.
\end{prop}

\begin{proof}
Recall the natural partial order on the set $S/\J$ of $\J$-classes of $S$ introduced in Section \ref{sec:prelim}.
Let $K$ be a minimal non-regular $\J$-class of $S$.
By \ref{it:J4}, $K$ is not the minimal $\J$-class of $S$, i.e. the set
\[
I:=\{ x\in S\colon J_x< K\}
\]
is non-empty. It is easy to see that $I$ is an ideal of $S$. By the choice of $K$, all elements of $I$ are regular.
Next let 
\[
L:= S\setminus (I\cup K).
\]
Note that $L$ may or may not be empty, may contain regular and non-regular elements, and that it is a union of $\J$-classes.
In this way, we have decomposed $S$ into the disjoint union
\[
S= L\,\dot{\cup}\, K\,\dot{\cup}\, I.
\]

For any set $M\subseteq \N_0$ with $0\in M$, let
    \[ P_{M} := (\{0\} \times L) \cup (M \times K) \cup (\Z \times I).\]
We will prove the proposition by showing the following:
\begin{enumerate}[leftmargin=10mm]
\item\label{it:PMsubd}
$P_M$ is a subdirect product of $\Z$ and $S$; and
\item\label{it:noniso}
If $P_{M_1}\cong P_{M_2}$ then $M_1=M_2$.
\end{enumerate}

Since for the remainder of the proof we will be simultaneously working with the $\J$ relations on different semigroups, we will distinguish them by means of superscripts.
Specifically, for a semigroup $U$, we write $\J^U$ for the $\J$ relation on $U$, and $J_u^U$ for the $\J^U$-class of $u\in U$.

\eqref{it:PMsubd}
To prove that $P_M\leq \Z\times S$, let $\alpha, \beta \in P_{M}$. We split our considerations 
into cases depending on which constituent part of $P_M$ each of $\alpha,\beta$ belongs to.

\textit{Case 1: at least one of $\alpha$ or $\beta$ is an element of $\Z \times I$.} 
Then $\alpha \beta \in \Z \times I$, since $I$ is an ideal of $S$. 

\textit{Case 2: $\alpha, \beta \in \{0\} \times L$.} 
Then $\alpha \beta \in \{0\} \times S\subseteq P_M$.

\textit{Case 3: $\alpha, \beta \in M \times K$.}
Suppose $\alpha=(m_1,k_1)$, $\beta=(m_2,k_2)$.
Since $K$ is a non-regular $\J^S$-class, we have $J_{k_1k_2}^S< K$ by \ref{it:J5}, i.e. $k_1k_2\in I$.
Therefore 
\[
\alpha\beta=(m_1+m_2,k_1k_2)\in \Z\times I\subseteq P_M.
\]
 
\textit{Case 4: one of $\alpha$ or $\beta$ belongs to $\{0\}\times L$, and the other to $M\times K$.}
Let us assume that $\alpha = (0, l) \in \{0\} \times L$ and $\beta = (m, k) \in M \times K$; the other case is symmetrical.
 Then $\alpha\beta = (m, lk)$. 
 Note that $J_{lk}^S\leq J_k^S=K$, and hence $lk\in K\cup I$.
 If $lk\in K$ then $(m,lk)\in M\times K\subseteq P_M$, while if $lk\in I$ then $(m,lk)\in \Z\times I\subseteq P_M$.
 
Hence indeed $P_M\leq \Z\times S$, and it remains to show that $P_M$ is subdirect.
Every integer appears as the first coordinate of some pair of $P_{M}$, because $\mathbb{Z} \times I\subseteq P_M$. Similarly, every element of $S$ appears as the second coordinate of some pair of $P_{M}$, because $S$ is the disjoint union of $L$, $K$, and $I$. 
This completes the proof of \eqref{it:PMsubd}.

\eqref{it:noniso}
We begin by characterising the $\J^{P_M}$-classes:

\begin{claim}
\label{cl:JPM}
For $(a,x),(b,y)\in P_M$ we have
\[
(a,x) \J^{P_M} (b,y)\quad  \Leftrightarrow\quad x\J^S y\text{ and } (a=b\text{ or }x,y\in I).
\]
\end{claim}

\begin{proof}
($\Rightarrow$)
Suppose $(a,x) \J^{P_M} (b,y)$. Identifying $P_M^1$ with $P_M \cup \{(0,1)\}$, where $1$ denotes the identity element of $S^1$, we can write
\[
(c_1,z_1)(a,x)(c_2,z_2)=(b,y)\quad\text{and}\quad (d_1,u_1)(b,y)(d_2,u_2)=(a,x),
\]
with $(c_1,z_1),(c_2,z_2),(d_1,u_1),(d_2,u_2)\in P_M^1$.
Equating the second components we obtain $x\J^S y$.
If $a=b$ there is nothing further to prove. Otherwise, suppose without loss that $a>b$.
Then from $c_1+a+c_2=b$ we have that at east one of $c_1$ or $c_2$ is negative. Suppose without loss that $c_1<0$.
This means that $z_1\in I$. Since $I$ is an ideal, it follows that $y=z_1az_2\in I$. Finally, $x\J^S y$ now implies that $x\in I$ as well.

($\Leftarrow$)
Since $x\J^S y$ we can write
$z_1xz_2=y$ and $u_1yu_2=x$ with $z_1,z_2,u_1,u_2\in S^1$. 
First suppose $a=b$. Note that $(0,z_1),(0,z_2),(0,u_1),(0,u_2)\in P_M^1$, and that
\[
(0,z_1)(a,x)(0,z_2)=(b,y)\quad\text{and}\quad (0,u_1)(b,y)(0,u_2)=(a,x),
\]
implying $(a,x)\J^{P_M} (b,y)$.
Now suppose that $x,y\in I$. Recall that this means that $J_x^S=J_y^S< K$. Since $K$ is a minimal non-regular $\J^S$-class it follows that $J_x^S$ is a regular $\J^S$-class.
By \ref{it:J3} we have that $z_1,z_2,u_1,u_2$ can be chosen to be in $J_x^S$ as well,
which in turn implies that $(b-a,z_1),(0,z_2),(a-b,u_1),(0,u_2)\in P_M$.
Now we have
\[
(b-a,z_1)(a,x)(0,z_2)=(b,y)\quad\text{ and } (a-b,u_1)(b,y)(0,u_2)=(a,x).
\]
and thus $(a,x)\J^{P_M} (b,y)$, completing the proof of the claim.
\end{proof}

Now suppose that $\phi : P_1 \to P_2$ is an isomorphism, where for brevity we write $P_i:=P_{M_i}$.
We proceed via a sequence of claims, in which we analyse how $\phi$ maps elements of $P_1$ of different forms.

\begin{claim}
\label{cl:phioxS}
$\phi(\{0\}\times S)=\{0\}\times S$.
\end{claim}

\begin{proof}
$\{0\}\times S$ is precisely the set of elements of finite order in both $P_1$ and $P_2$.
\end{proof}

\begin{claim}
\label{cl:phiZxI}
$\phi(\Z\times I)=\Z\times I$.
\end{claim}

\begin{proof}
By Claim \ref{cl:JPM}, $\Z\times I$ is precisely the set of elements whose $\J$-classes are infinite in both $P_{M_1}$ and $P_{M_2}$.
\end{proof}

\begin{claim}
\label{cl:phiMK}
$\phi\bigl((M_1\setminus\{0\})\times K\bigr)=(M_2\setminus\{0\})\times K$.
\end{claim}

\begin{proof}
Claim \ref{cl:phiZxI} implies that
\[
\phi\bigl((\{0\}\times L)\cup (M_1\times K)\bigr)=(\{0\}\times L)\cup (M_2\times K).
\]
But, for $i=1,2$,  the set of elements of infinite order in $(\{0\}\times L)\cup (M_i\times K)$
is precisely $(M_i\setminus\{0\})\times K$.
\end{proof}

\begin{claim}
\label{cl:phike}
For every $x\in I$ and every $k\in\Z$ we have
\[
\phi(k,x)=(\epsilon k,x') \quad \text{for some } \epsilon=\pm1,\ x'\in I.
\]
\end{claim}

\begin{proof}
Let $x\in I$ be fixed.
We will analyse the effect of $\phi$ on the $\J^{P_1}$-class of $(0,x)$, which, by Claim \ref{cl:JPM}, is equal to $\Z\times J_x^S$.
Certainly, by Claims \ref{cl:phioxS}, \ref{cl:phiZxI}, we  have
\begin{equation}
\label{eq:oy}
\forall y\in J_x^s\colon \exists y'\in I\colon \phi(0,y)=(0,y').
\end{equation}
Since $J_x^S\subseteq I$ and $I$ consists solely of regular elements, it follows by \ref{it:J2} that $J_x^S$ must contain an idempotent $e$.
By \eqref{eq:oy} we have $\phi(0,e)=(0,e')$ for some idempotent $e'\in I$.
Now suppose that $\phi(1,e)=(a,e'')$, where $a\in\Z$ and $e''\in I$.
Consider an arbitrary $y\in J_x^S$ and $k>0$.
Write $y=uev$ with $u,v\in J_x^S$, which can be done by \ref{it:J3}.
By \eqref{eq:oy} we have
\[
\phi(0,u)=(0,u'),\ \phi(0,v)=(0,v')\quad \text{for some } u',v'\in I.
\]
Then
\begin{equation}
\label{eq:ky}
\begin{aligned}
\phi(k,y) &= \phi \bigl((0,u)(k,e)(0,v)\bigr)=\phi\bigl((0,u)(1,e)^k(0,v)\bigr)=
\phi(0,u)\bigl(\phi(1,e)\bigr)^k\phi(0,v)\\
&=(0,u')(a,e'')^k(0,v')=(0,u')\bigl(ak,(e'')^k\bigr)(0,v')=(ak,y'),
\end{aligned}
\end{equation}
where $y':=u'(e'')^kv'$.
If $k<0$, a similar reasoning proceeding from $\phi(-1,e)$ instead of $\phi(1,e)$, yields
\begin{equation}
\label{eq:ky1}
\phi(k,y)=(bk, y''),
\end{equation}
for some $b\in\Z$ and  $y''\in I$.

Now, \eqref{eq:oy}, \eqref{eq:ky}, \eqref{eq:ky1} entirely describe the effect of $\phi$ on the $\J^{P_1}$-class $\Z\times J_x^S$.
By Claim~\ref{cl:JPM}, its image must be of the form $\Z\times J_z^S$ for some $z\in I$.
Hence, looking at the first components of the right-hand sides in \eqref{eq:oy}, \eqref{eq:ky}, \eqref{eq:ky1} we must see all integers.
This can happen only if $\{a,b\}=\{\pm 1\}$.
The claim follows by setting $y=x$ in  \eqref{eq:oy}, \eqref{eq:ky}, \eqref{eq:ky1}, and setting $\epsilon$ to be $a$ or $b$ depending on whether $k\geq 0$ or $k<0$.
\end{proof}

\begin{claim}
\label{cl:phimx}
For every $m\in M_1$ and every $x\in K$ we have
\[
\phi(m,x)=(m,x') \quad\text{for some } x'\in S.
\]
\end{claim}

\begin{proof}
Suppose $\phi(m,x)=(a,x')$.
By choice of $K$, we have $x^2\in I$.
Therefore $\phi(2m,x^2)=(2m,x'')$ for some $x''\in I$ by Claim~\ref{cl:phike}.
Now we have
\[
(2m,x'')=\phi(2m,x^2)=\phi\bigl((m,x)^2\bigr)=\bigl(\phi(m,x)\bigr)^2=(a,x')^2=\bigl(2a,(x')^2\bigr),
\]
from which $a=m$, as claimed.
\end{proof}

Claims \ref{cl:phiMK} and \ref{cl:phimx} together give $M_1=M_2$, completing the proof of \eqref{it:noniso}, and of the proposition.
\end{proof}

\section{Conclusion}

In proving that there are countably many subdirect products in $\Z\times S$ when $S$ is regular in Section \ref{sec:Sreg}, the assumption that the subsemigroup $T$ is a subdirect product was only used to establish that all the sets $T_x$ are non-empty.
One may therefore wonder whether perhaps a stronger property is also satisfied, namely that $\Z\times S$ has only countably many
\emph{subsemigroups}.
This, however, is not true. For if $S$ is a regular semigroup with a non-regular subsemigroup $S_0$, then, by our Main Theorem, there are uncountably many pairwise non-isomorphic subdirect products in $\Z\times S_0$, and they are all, of course, subsemigroups of $\Z\times S$.

In fact, we can give a characterisation for when $\Z\times S$ has only countably many subsemigroups.
To state it, we need the notion of semigroups that are \emph{unions of groups} (a.k.a.\ \emph{completely regular semigroups}).
These are semigroups in which every element belongs to a subgroup; for more details see \cite[Section 4.1]{Ho:FST}.
Certainly, every union of groups is a regular semigroup.

\begin{cor}
\label{co:subs}
The following are equivalent for a finite semigroup $S$:
\begin{thmenumerate}
\item \label{it:co1}
$S$ is completely regular;
\item\label{it:co2}
$\Z\times S$ contains only countably many subsemigroups;
\item\label{it:co3}
$\Z\times S$ contains only countably many subsemigroups up to isomorphism.
\end{thmenumerate}
\end{cor}

\begin{proof}
\ref{it:co1}$\Rightarrow$\ref{it:co2}
Suppose $S$ is completely regular. 
It is again sufficient to prove that every subsemigroup $U$ of $\Z\times S$ is finitely generated.
Let $Z'$ and $S'$ be the projections of $U$ to $\Z$ and $S$ respectively. Then $Z'\leq\Z$, $S'\leq S$, and $U$ is a subdirect product in $Z'\times S'$.
We consider different options for $Z'$.
If $Z'=\{0\}$ then $U$ is finite.
If $Z'$ is a non-trivial subgroup of $\Z$ then it is isomorphic to $\Z$,  and hence $U$ is finitely generated by Theorem~\ref{th:fg}.
Suppose now that $Z'\leq \N_0$.
If $0\not\in Z'$ then in fact $Z'\leq\N$, and hence $U$ is finitely generated by the proof \cite[Theorem D, (iii)$\Rightarrow$(i)]{acnr1}.
If $0\in Z'$ then $U=U_0\cup U_1$, where $U_0:=U\cap (\{0\}\times S)$ and $U_1=U\cap (\N\times S)$.
But $U_0$ is finite, and $U_1$ is finitely generated by the above argument, and hence $U$ itself is finitely generated.
Finally, if $\Z\leq\N_0$ then the assertion follows from the previous case and $-\N_0\cong \N_0$.

\ref{it:co2}$\Rightarrow$\ref{it:co3}
This is obvious.

\ref{it:co3}$\Rightarrow$\ref{it:co1}
We prove the contrapositive. Suppose $S$ is not completely regular. 
Let $s\in S$ be an element of $S$ that does not lie in a subgroup of $S$.
This means that $s\not\in \{ s^k\colon {k\geq 2}\}$, and hence
the monogenic subsemigroup $\langle s\rangle\leq S$ is not regular. Therefore, the Main Theorem gives that
there are uncountably many pairwise non-isomorphic subdirect products in $\Z\times \langle s\rangle$.
All of them are subsemigroups of $\Z\times S$, and the proof is complete.
\end{proof}

Putting side by side Corollary \ref{co:subs} and  \cite[Theorem D]{acnr1} we obtain the following curious fact:

\begin{cor}
Let $S$ be a finite semigroup. Then $\N\times S$ has only countably many subsemigroups (up to isomorphism) if and only if $\Z\times S$ has only countably many subsemigroups (up to isomorphism).\hfil \qed
\end{cor}

The analogous statement for subdirect products instead of subsemigroups is not true: compare again the Main Theorem and 
\cite[Theorem E]{acnr1}.

Based on our findings in this paper, as well as those of \cite{acnr1}, we ask the following questions:

\begin{nitemize}
\item
Is it true that if $S$ is a finite regular semigroup, then there are only countably many subdirect products in any $G\times S$, where $G$ is a finitely generated abelian group.
\item
Characterise all finite semigroups $S$ with the property that for every finitely generated commutative semigroup $C$ there are only countably many subdirect products in $C\times S$.
\item
Let $U$ be the bicyclic monoid or the free monogenic inverse monoid. Describe all finite semigroups $S$ such that there are only countably many subdirect products in~$U\times S$.
\end{nitemize}

\bibliographystyle{plain}

\end{document}